\documentclass[11pt,psamsfonts]{amsart}
\usepackage{a4} 
\usepackage[english]{babel}
\usepackage{amsmath,amssymb,graphicx,enumerate,bbm}

\usepackage{amsmath}
\usepackage{amsthm}
\usepackage{amssymb}
\usepackage{amscd}
\usepackage{amsfonts}
\usepackage{amsbsy}
\usepackage{epsfig,afterpage}
\usepackage{psfrag}
\usepackage{overpic}
\usepackage{latexsym}
\usepackage{pstcol}
\usepackage{graphicx}
\usepackage[all]{xy}

\newcommand{\R}{\ensuremath{\mathbb{R}}}

\newcommand{\la}{\lambda}

\newcommand{\e}{\varepsilon}

\newcommand{\sgn}{\operatorname{sgn}}
\newtheorem {theorem} {Theorem} [section]
\newtheorem {proposition}  {Proposition}[section]
\newtheorem {corollary}  {Corollary}[section]

\newtheorem {definition}  {Definition}[section]

\begin{document}

\title [Nonlinear Sliding and Singular Perturbation]{Nonlinear Sliding of Discontinuous Vector Fields and Singular Perturbation}

\author[  da Silva, Meza-Sarmiento, Novaes]
{P. R. da Silva $^1$,  I. S. Meza-Sarmiento $^1$ and D. D. Novaes $^2$ }

\address{$^1$  Departamento de Matem\'{a}tica --
IBILCE--UNESP, Rua C. Colombo, 2265, CEP 15054--000 S. J. Rio Preto,
S\~ao Paulo, Brazil}

\address{$^2$  IMECC -- UNICAMP, CEP 13081Ð970, Campinas, S\~{a}o Paulo, Brazil}

\email{prs@ibilce.unesp.br}
\email{isofia1015@gmail.com}
\email{ddnovaes@gmail.com}



\maketitle

\begin{abstract}
We consider piecewise smooth vector fields (PSVF) defined in open sets $M\subseteq\R^n$ with switching manifold being a smooth surface $\Sigma$. 
The PSVF are given by pairs $X = (X_+, X_-)$, with  $X = X_+$  in $\Sigma_+$ and $X = X_-$  in $\Sigma_-$ where $\Sigma _+$ and  $\Sigma _-$
are the regions on $M$ separated by $\Sigma.$
A regularization of $X$  is a 1-parameter family  of smooth vector fields $X^{\e},\e>0,$ 
satisfying that $X^{\e}$ converges pointwise to $X$ on $M\setminus\Sigma$, when $\e\rightarrow 0$.
Inspired by the Fenichel Theory \cite{F}, the sliding and sewing dynamics on the discontinuity locus $\Sigma$ can be defined
as some sort of limit of the dynamics of a nearby smooth regularization $X^{\e}$.
While the  linear regularization requires that for every $\e>0$ the regularized field $X^{\e}$  is in the convex combination of $X_+  $ and $X_- $ the  nonlinear regularization requires only that $X^{\e}$ is in a continuous combination of $X_+  $ and $X_- $.
We prove that for both cases,  the sliding dynamics on $\Sigma$ is determined by the reduced dynamics  on the critical manifold  of a singular perturbation problem.
\end{abstract}

\section{Introduction}\label{s1}
Piecewise-smooth vector fields have been investigated at least since 1930. This kind of systems is present 
in many physical phenomenons, for instance, collisions between rigid bodies, in mechanical friction and impacts, 
electrical circuits and so ones. They also appear in control theory, economy, cell mitosis, predator-prey and  
climate problems, etc. See \cite{dBeBCK,AF} for a general scope of these topics.

A \textit{piecewise-smooth vector field (PSVF)}    is determined by three elements: a set $M\subset\R^n$, a switching set  $\Sigma\subset M$ and
a vector field $X:M\rightarrow \R^n$. The simplest case is when  $M$ is an open set  and $\Sigma=h^{-1}(0)$ 
for some smooth function $h:M\rightarrow\R$ having $0$ as a regular value. So $\Sigma$ 
divides $M$  in two regions $\Sigma^+, \Sigma^-$ and   
\begin{equation}
X=\left(\dfrac{1+\sgn(h)}{2}\right)X_+ +\left(\dfrac{1-\sgn(h)}{2}\right)X_-,
\label{pds}
\end{equation}
for some smooth vector fields $X^+$ and $X^-$on 
$M$.

The regularization method which we consider keeps the phase-portrait unchanged outside $O(\e)$-neighborhoods of the discontinuity set and it involves transition functions. A regularization of a PSVF $X:M\rightarrow \R^n$ is a 1-parameter family  of smooth vector fields $X^{\e}:M\rightarrow \R^n,\e>0,$ satisfying that 
$X^{\e}$ converges pointwise to $X$ on $M\setminus\Sigma$, when $\e\rightarrow 0.$

We denote by  $[X_-,X_+]$ the \emph{convex combination} of
$X_-$ and $X_+$:
\[[X_-,X_+]=\{(1/2+\lambda/2) X_+ +(1/2-\lambda/2 ) X_-: \lambda\in[-1,1]\}.\]

 We say that a regularization $X^{\e}:M\rightarrow \R^n$  of $X$ is \textit{linear} if $X^{\e}(p)\in[X_-(p),X_+(p)]$, for any $p\in M.$

The main example of linear regularization is the Sotomayor-Teixeira regularization proposed  in \cite{LT,ST}. 
It bases on replacing the two adjacent fields by an $\e$-parametric field  built as a linear convex combination 
of them in an $\e$-neighborhood of the discontinuity. 
A \emph{transition function} $\varphi:\R\rightarrow\R$ is used, that is, a function satisfying
$\varphi(x)=1$ for $x\geq1$, $\varphi(x)=-1$ for $x\leq1$ and $\varphi'(x)>0$ if $x\in(-1,1)$, in 
order to get a family of continuous vector fields that approximates the discontinuous one. 

\begin{definition}
The $\varphi$--Sotomayor-Teixeira regularization of  $X_-$ and $X_+$  is
the one parameter family 
\[X^{\e}= ( 1/2  + \varphi(h/\e)/2 )X_+ +  ( 1/2  - \varphi(h/\e)/2 )  X_-.\]
\end{definition}

Note that it is necessary that $X_+$ and $X_-$ to be defined on both of sides of $\Sigma$.

Different regularizations can lead to different ways of defining the sliding solutions on $\Sigma$. 
The way chosen will depend on suitability to model the problem.  
Linear regularizations are not sufficiently general for physical or biological switching processes, 
such as mechanical chatter or electrical heating, or the energy required to activate the switch at $h=0$. 
A vast expanse of non-equivalent but no less valid dynamical systems can be unconsidered.

In \cite{J} the authors considers a general model of discontinuous dynamics in terms of nonlinear sliding modes,
along with its perturbation by smoothing and its response to errors, and applied it to a heuristic model that 
captures some key characteristics of dry friction.  

A large bibliography can be found about regularization of discontinuous systems.  We refer \cite{BRT, KH1,KH2,LST,LST2,LST3,LST4,NJ} for instance.

Our  purpose  is to study  the sliding dynamics emerging from nonlinear regularization.

 A  \emph{continuous combination} of $X_-$ and $X_+$ is a 1--parameter family of smooth vector fields 
 $\widetilde{X}(\lambda, .)$,  with $\lambda\in [-1,1]\,$
 and satisfying that
\[\widetilde{X}(-1,p)=X_-(p),\quad \widetilde{X}(1,p)=X_+(p),\forall p\in M.\]

\begin{definition}
A regularization $X^{\e}$ is of the kind nonlinear if  there exists a continuous combination $\widetilde{X}$ 
such that $X^{\e}(p)\in \{\widetilde{X}(\lambda,p), \lambda \in [-1,1]\} $  for any $p\in M.$ 
\end{definition}

In the same way we define the \emph{ $\varphi$--nonlinear regularization} of  $X_-$ and $X_+$. 

\begin{definition}
A $\varphi$--nonlinear regularization of  $X_-$ and $X_+$ is the $1$--parameter family given by $X^{\e}(p)=\widetilde{X}(\varphi ( h/\e),p).$ 
\end{definition}
Note that  if $h>\e$ then $ \varphi (h/\e)=1$ and $X^{\e}= X_+$; and  if $h< -\e$ then $ \varphi (h/\e)=1$ and $X^{\e}= X_-$.

 \begin{figure}
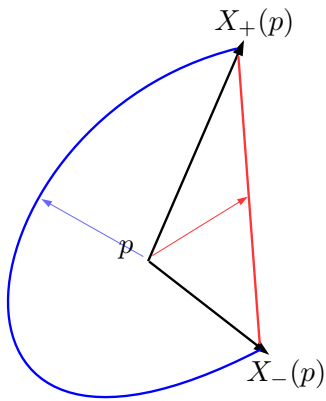

\begin{center}
\begin{overpic}[width=6cm]{convexAcont.pdf}
\put(30,50){$p$}\put(45,85){$X_+(p)$}\put(50,30){$X_-(p)$}
\end{overpic}
\end{center}
\caption{ Linear (red)  and  nonlinear (blue) regularizations.}
\end{figure}

The paper is organized as follows. In Section \eqref{s2} we recall the usual definitions
of sewing and sliding  regions and generalize them for the nonlinear case. We compare 
the classical and the new sliding regions. We also define the nonlinear sliding vector field.
Finally we state the main theorem which establishes a connection between the nonlinear sliding 
vector fields and singular perturbation problems. In Section \eqref{s3} we prove our main result 
and present an illustrative  example. In Sections \eqref{s4} and \eqref{s5} we apply our techniques in the description
of the nonlinear regularization of  normal forms of PSVF in $\R^2$ and in $\R^3$ .

\section{Preliminaries and Main Result}\label{s2}
Let $X$ be  a PSVF like \eqref{pds}. We use the notation $X(p)=(X_+(p), X_-(p))$ or simply $X=(X_+,X_-)$ and assume 
that  $X$ is bi-valuated on the  \emph{switching manifold} $\Sigma = h^{-1}(0)$. We may not have unicity of trajectories by points $p\in\Sigma$.

Denote \[ X.h(p)=\nabla h(p).X(p).\]

The points $p\in\Sigma$ are classified   as {\emph{regular} if  $(X_-. h)(p)\neq 0$ and $ (X_+. h)(p)\neq 0;$
and as \emph{singular} if $(X_-. h)(p)=0$ or $(X_+. h)(p)= 0.$ 

There are two kinds of  singular points: the equilibrium points of $X_-$ or $X_+$  on $\Sigma$ and 
the points where the trajectories of $X_-$ or $X_+$ 
are tangent to $\Sigma$. \

The regular points are classified as \textit {sewing} if $(X_-. h).(X_+. h)>0$ 
or \textit{sliding} if $(X_-. h).(X_+.h)<0.$  The sets of sewing and sliding  points are denoted by 
$\Sigma^{sw}$ and $\Sigma^{sl}$ respectively.  \

Consider the PSVF given by \eqref{pds},  a continuous combination $\widetilde{X}(\lambda,p)$ of $X_-$ and $X_+$ and a regular point $p\in\Sigma$.
\begin{itemize}
\item[(a)] We say that  $p$ is a \emph{nonlinear sewing} point and denote $p\in\Sigma^{sw}_{n}$ if  $(\widetilde{X}.h)(\lambda,p)\neq 0 $ for all $\lambda\in [-1,1].$
\item[(b)] We say that  $p$ is a \emph{nonlinear slidding} point and denote $p\in\Sigma^{sl}_n$ if  there exists $\lambda\in [-1,1],$ such that   $(\widetilde{X}.h)(\lambda,p)=0$.
\end{itemize}

\begin{proposition} Consider a PSVF given by \eqref{pds}. We have $\Sigma^{sw}_n\subseteq\Sigma^{sw}$ and $\Sigma^{sl}\subseteq\Sigma^{sl}_n.$
\end{proposition}
\begin{proof} Suppose that $p\in\Sigma^{sw}_n$. Since $(\widetilde{X}.h)(p)\neq 0 $ for all $\lambda\in [-1,1]$  then both  $X_+(p)$ and  $X_-(p)$ point toward either 
$\Sigma^+$ or to $\Sigma^-$. Thus $(X_-. h).(X_+. h)>0$ in $p$.  It follows that  $p\in\Sigma^{sw}$.  Suppose now that $p\in\Sigma^{sl}.$
It implies that $X_+(p)$ and  $X_-(p)$ are directed to opposite sides. Thus all continuous path connecting them intersects $\Sigma$. It means that $p\in\Sigma^{sl}_n.$
\end{proof}

Consider the PSVF given by \eqref{pds} and a continuous combination of $X_-$ and $X_+$. For each $p\in \Sigma^{sl}_n$
there exists  $\lambda (p)\in [-1,1]$ such that  $(\widetilde{X}.h)(p)=0$. 
We say that $\widetilde{X}(\lambda (p),p)$ is a \emph{nonlinear sliding vector field}.\\

\noindent \textbf{Example 1.}  Let $X=(X_+,X_-)$ be a PSVF defined on $\R^2$ with $h(x,y)=y$,
$X_+=(1,1-x)$ and  $X_-=(1,3-x).$ Consider the continuous combination of $X_-$ and $X_+$  given by 
$\widetilde{X}=(-1+2\lambda^2, -\lambda+2\lambda^2-x)$.  Solving $-\lambda+2\lambda^2-x=0$ in the variable $\lambda$ we determine
two possible sliding vector fields
\[X^{\lambda_1}=\left((-3+4x+\sqrt{1+8x})/4,0\right),\quad X^{\lambda_2}=\left((-3+4x-\sqrt{1+8x})/4,0\right).\]

Note that $x\geq-1/8$  is a necessary  condition. Moreover, we also require $|\lambda_1|<1$ and $|\lambda_2|<1.$
It implies that $\Sigma^{sl}=(1,3)$ and $\Sigma^{sl}_n=[-1/8,1)\cup(1,3)$.  In $[-1/8,1)$ are defined two  nonlinear sliding vector fields  ($X^{\lambda_1}$
and $X^{\lambda_2}$ ) and on $(1,3)$ is defined only $X^{\lambda_2}.$ \label{ex2}\

Consider a PSVF as given by \eqref{pds}, a transition function $\varphi$ and  the 
$\varphi$-- nonlinear regularization of  $X_-$ and $X_+$.  Our main result is the following.

\begin{theorem}\label{mainteo}Consider a PSVF as given by \eqref{pds} with a continuous combination $\widetilde{X}$ of $X_+$ and $X_-$ and  nonlinear sliding region $\Sigma^{sl}_n$.
There exists a singular perturbation problem 
\begin{equation} r\dot{\theta}=\alpha(r,\theta,p)=-\sin\theta\alpha_0(r,\theta,p),\quad \dot{p}=\beta(r, \theta,p),\label{eqteo}\end{equation}
with $r\geq0,\theta\in[0,\pi],p\in\Sigma^{sl}_n$ and critical manifold $\mathcal {S}$ satisfying the following.
\begin{itemize}
\item[(a)] For any normally hyperbolic  $(\theta,p)\in\mathcal{S}$ there exist homeomorphic neighborhoods $p\in I_p\subset \Sigma^{sl}_n$ and $\mathcal{S}_p\subset \mathcal {S}$
and a sliding vector field $\widetilde{X}(\lambda (p),p)$ defined in $I_p$ which   is conjugated to the slow flow of \eqref{eqteo} on $\mathcal{S}_p\subset \mathcal {S}$.
\item[(b)] For any $p\in int(\Sigma^{sl}_n)$ consider $\ell=\# \{\theta\in(0,\pi):(\theta,p)\in\mathcal{S}\}. $ There exist $\ell$ choices of sliding vector fields defined in $p$.
\item[(c)] If all points on $\mathcal{S}$ are normally hyperbolic then there exists only one choice for the sliding  vector field in  $\Sigma^{sl}_n$.
\end{itemize}
\end{theorem} 

The  zero set $\mathcal{S}=\left\{(\theta,p):\alpha_0(0,\theta,p)=0\right\}$ is called \emph{slow manifold} 
and  a point $q=(\theta,p)\in\mathcal {S}$ is \textit{normally hyperbolic} is $\dfrac{\partial\alpha_ 0}{\partial\theta}(0,q)\neq0.$\\

 If we consider a time-rescaling, system \eqref{eqteo} becomes
\begin{equation} 
\theta'=\alpha(r,\theta,p),\quad p'=r\beta(r,\theta,p)
 \label{fast}\end{equation}
 
In general we refer to systems \eqref{eqteo} and \eqref{fast} as \emph{ slow}  and \emph{fast} systems respectively.  
If $r=0$ in \eqref{eqteo} we have the \emph{reduced} system and if $r=0$ in \eqref{fast} we have the \emph{layer} system.
Note that $\mathcal {S}$ is a set of equilibrium points of the layer system which has a vertical flow.
Slow and fast systems are equivalent for $\e>0$, that is, they have the same phase portrait.
It means that the phase portrait of  system \eqref{eqteo}  for $\e\sim 0$ approaches two limit
situations: the phase portrait of reduced and layer systems.

Our theorem establishes a connection between a singular perturbation problem and the sliding system in $\Sigma$.
We obtain that the possible sliding dynamics are identified with  reduced dynamics on the slow manifold of the singular perturbation problem. 
More specifically, the possible sliding systems are projections ($\mathcal{S}\rightarrow \Sigma $) of the slow flow.
This is only an initial step in the analysis of the regularization. 
The advantage of our approach is that we can use  GSP-theory techniques. However, the major obstacle to the overall understanding of the regularized system 
is the existence of non--normally hyperbolic points in the slow manifold. In this case additional efforts need to be made.
See for instance \cite{DR, J1, KS}.

In \cite{PS} the autors study the regularization  of an oriented 1-foliation $\mathcal{F}$ on $M \setminus \Sigma$ where $M$ is a  smooth manifold 
and $\Sigma \subset M$ is a closed subset, which can be interpreted as the discontinuity locus of $\mathcal{F}$.

\section{Proof of the Main Theorem} \label{s3}
In this section we prove our main result and consider an illustrative example.

\begin{proof}[Proof of the main Theorem.]
Take a PSVF $X$ and a continuous combination $\widetilde{X}$  as in the statement. Choose local coordinates $(x,y)=(x_1,...,x_{n-1},y)$ 
such that $h(x,y)=y$ and $\widetilde{X}(\lambda,x,y)=(s_1(\lambda,x,y),....,s_n(\lambda,x,y))$. Thus $\Sigma=\{p=(x,0)\}$ and
\[\Sigma^{sl}_n=\{p=(x,0)\in\Sigma:\exists\lambda\in(-1,1),\widetilde{X}(\lambda,x,0)(0,...,0,1)=s_n(\lambda,x,0)=0\}.\]
Consider a strictly increasing transition function $\varphi$, that is, it tends to $1$ and $-1$ without actually achieve these values, for instance $\varphi(s)=\tanh(s)$.
 The $\varphi$--nonlinear regularization is given by
\[\widetilde{X}(\varphi(y/\e),x,y)=(s_1(\varphi(y/\e),x,y), ..., s_n(\varphi(y/\e),x,y)).\]
Its  trajectories  are determined by the system
\begin{equation}x_i'=s_i(\varphi(y/\e),x,y)\quad y'=s_n(\varphi(y/\e),x,y).,\quad i=1,...,n-1.\label{singX}\end{equation}
The nonlinear sliding vector field is $\widetilde{X}(\lambda,x,y)$ with $\lambda$ satisfying that $s_n(\lambda,x,y)=0.$\

Consider the blow up $y=r\cos\theta$ and $\e=r\sin\theta$ with $r \geq 0$ and $\theta\in[0,\pi].$
Denote $\psi(\theta)=\varphi(\cot\theta)$ which  is an injective decreasing function with $\psi(0)=1,\psi(\pi)=-1$. Thus
the system \eqref{singX} become
\[r\dot{\theta}=-\sin\theta s_n(\psi (\theta),x,r\cos\theta),\quad \dot{x_i}=s_i(\psi (\theta),x,r\cos\theta), \quad i=1,...,n-1.\]
Then  \[\alpha_0 (r,\theta,x)=s_n(\psi (\theta),x,r\cos\theta)\] and \[\beta(r,\theta,x)=(s_1(\psi (\theta),x,r\cos\theta), ..., s_{n-1}(\psi (\theta),x,r\cos\theta))\]
determine the singular perturbation desired.\\
\noindent \emph{(a).} Note that slow manifold and nonlinear sliding region are defined by the same equation : $s_n(\psi (\theta),x,0)=0$ for the slow manifold and
 $s_n(\lambda,x,0)=0$ for the sliding. Thus the local diffeomorphism  is immediate because $\psi (\theta)$ is decreasing in $(-1,1)$. Moreover the slow flow is determined 
 by the reduced system $\dot{x_i}=s_i(\psi(\theta),x,0)$ and the nonlinear sliding vector field by  $\dot{x_i}=s_i(\lambda,x,0)$. Since $\psi(\theta)$ and $\lambda$ have the same expression 
 it concludes (a).\\
 \noindent \emph{(b).} The number of possible choices of nonlinear sliding vector fields is the number of
 possible choises of $\lambda$ such that $s_n(\lambda,x,0)=0$. Since the number of choices of $\lambda$ is 
 exactly the same that the number of $\theta=\theta(x,0)$ defined implicitly by $s_n(\psi(\theta),x,0)=0$ , the statement (b) is proved.\\
  \noindent \emph{(c).}  The normal hyperbolicity of the points on $\mathcal{S}$ implies that the graphic 
  implicitly defined by $s_n(\psi(\theta),x,0)=0$ is the graphic of only one $\theta=\theta(x)$. It means that 
  $s_n(\lambda,x,0)=0$ defines uniquely $\lambda=\lambda(x).$  So the statement (c) is proved. 
  \end{proof}
  
\begin{figure}[h!]
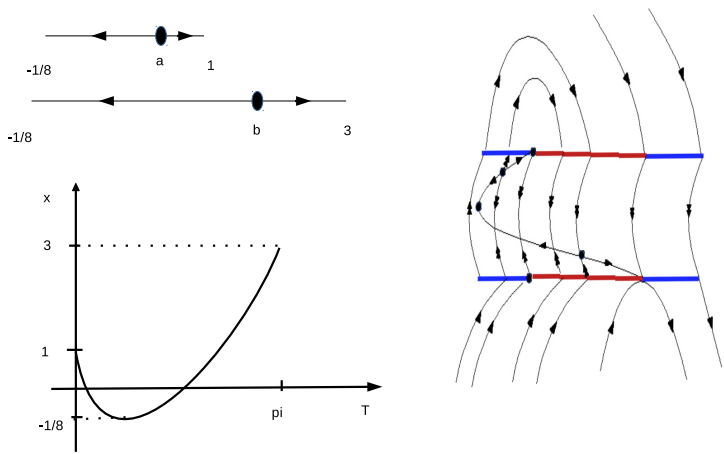

\epsfysize=8cm \centerline{\epsfbox{svf1.pdf}\epsfbox{figDP1a.pdf}}\caption{\small{The nonlinear sliding vector field $X^{\lambda_1}$ with $a=1-\sqrt{2}/2$ and 
the nonlinear sliding vector field $X^{\lambda_2}$ with $b=1+\sqrt{2}/2.$ Sewing (blue) and sliding (red) regions and the blow up of a nonlinear regularization of 
$X_-$ and $X_+$  producing a 
nonlinear sliding.}}\label{figsvf1}
\end{figure}

We could apply the direcional blow up in the proof of the main theorem. The direcional blow up consists in the
following change of coordinates:
\[\Gamma:\R^{n+1}\rightarrow \R^{n+1},\quad
(x,y,\e)=\Gamma(x,\overline{y},\overline{\e})=(x,\overline{\e}.\overline{y},\overline{\e}).\] This blow up was considered by the authors in 
\cite{BPT} and in \cite{NJ}. 
However, geometrically speaking, it is more convenient to consider
the  polar blow up coordinates

\[\Lambda:[0,+\infty)\times[0,\pi]\times\R^{n-1}\rightarrow\R^{n+1},\quad (x,y,\e)=\Lambda(r,\theta,x)=(x,r \cos \theta,r \sin \theta).\] 
The map $\Lambda$ induces the  vector field on $[0,+\infty)\times[0,\pi]\times\R^{n-1}$.
The parameter value $\e=0$ is now represented by $r=0$. We observe that the direcional blow up and the polar blow
up are essentially the same. In fact, if we consider the map
$G:C=[0,+\infty)\times[0,\pi]\times\R^{n-1}\rightarrow\R^{n+1}$ given by
$G(r,\theta,\pi)=(\cot\theta,y,r\sin\theta)$ then $\Gamma\circ
G=\Lambda.$

\[\xymatrix{
C \ar[r]^{\Lambda}
\ar[d]_G & \R^{n+1}\\
\R^{n+1} \ar[ru]_{\Gamma} }\]

 \noindent\textbf{ Example 2. } Let $X=(X_+,X_-)$ as in Example 1 We consider the nonlinear regularization given by $X^{\e}=\widetilde{X}(\varphi (y/\e),x,y)$.
The trajectories of $X^{\e}$ satisfies the system
\[x'=-1+2\varphi(y/\e)^2,\quad y'=-\varphi(y/\e)+2\varphi(y/\e)^2-x.\]
Consider the blow up $y=r\cos\theta$ and $\e=r\sin\theta$ with $r\geq 0$ and $\theta\in[0,\pi]$. Thus, denoting $\psi(\theta)=\varphi(\cot\theta)$, the system becomes
\[r\dot{\theta}=-\psi(\theta)+2\psi(\theta)^2-x,\quad \dot{x}=-1+2\psi(\theta)^2.\]
The slow manifold is given by $x=-\psi(\theta)(2\psi(\theta)-1).$
It is easy to see that $x$ is a smooth curve connecting $(\theta,x)=(0,1)$ and $(\theta,x)=(\pi,3)$. Moreover the derivative
of $x$ is $x'=\psi'(4\psi-1)$ and it is zero if $\psi=\dfrac{1}{4}$. In this case $x=-1/8$. The slow flow is given by
$x'=(-3+4x\pm\sqrt{1+8x})/4$
which is exactly the same expression of the sliding $X^{\lambda_1}$ and $X^{\lambda_2}.$ See figure \eqref{figsvf1}.

\section{Nonlinear regularization of Generic PSVF's on $\R^2$}\label{s4}

In this section normal forms of generic PSVF 
(sewing, sliding, saddle and fold regular) on $\R^2$  are discussed.
We  use singular perturbation techniques to analyze the dynamics of their $\varphi$--nonlinear regularizations.
From now on we are considering the continuous combination
\begin{equation}\label{contcomb}\widetilde{X}=\dfrac{1+\la}{2}X_+(x,y)+\dfrac{1-\la}{2}X_-(x,y)+(P(\la),Q(\la))\end{equation}
with $P,\,Q$ real polynomials of degree $m$ and 
 $n$ respectively, satisfying the condition $P(\pm 1)=0$ and $Q(\pm 1)=0$.\\
 
 We recall that $\varphi$ denotes a strictly increasing transition function and $\psi(\theta)=\varphi(\cot\theta)$.

\subsection{Sewing}

We say that $X_N=(X_+,X_-)$  is the normal form of the sewing PSVF defined on $\R^2$ if   $X_+(x,y)=(a,b)$, 
$X_-(x,y)=(c,d)$, $a,\,b,\,c, \,d\in\R\setminus\{0\}$, $bd>0$, $h(x,y)=y$ and
 $\Sigma=h^{-1}(0)$. In this case we have $\Sigma=\Sigma^{sw}$. \

\begin{proposition} Let $X_N=(X_+,X_-)$ be the normal form of the sewing PSVF defined on $\R^2$ and consider
the continuous combination \[\widetilde{X}=(a/2+c/2+(a/2-c/2)\la+P(\la),b/2+d/2+(b/2-d/2)\la+Q(\la)).\] 
\begin{itemize}
\item [(a)] There exists a real continuous function $\Delta(\la)$ such that: If $\Delta$ has $k\leq n$ 
real roots $\la_{i}\in[-1,1]$, $i=1,\ldots,\,k$ then $\Sigma=\Sigma^{sl}_n$; if  $\Delta(\la)\neq0$,
 for all $\la\in[-1,1]$ then $\Sigma^{sw}_{n}=\Sigma$.
\item [(b)] For the case that $\Sigma^{sl}_{n}=\Sigma$ the singular perturbation problem \eqref{eqteo}
 in the Theorem \ref{mainteo} is such that the slow manifold is the union of $k\leq n$
  lines $\theta=\theta_i,\,i=1,\ldots,k$. Moreover the slow flow on $\theta=\theta_i$ is 
  determined by $\dot{x}=\sgn((a+c)/2+\psi(\theta_i)(a-c)/2+P(\psi(\theta_i))).$
\end{itemize}

\end{proposition}

\begin{proof}
Let $\widetilde{X}$ be the continuous combination of $X$ given by the statement
of the theorem. Let $\Delta$ be the real continuous function given by 
\begin{equation}\label{eqsliding}
\Delta=\Delta(\la):=(b/2+d/2)+(b/2-d/2)\la+Q(\la).
\end{equation} 
According with the definition $p\in\Sigma^{sl}_n$ if there exists $\lambda\in [-1,1],$ satisfying 
 $(\widetilde{X}.h)(p)=\Delta=0$. Thus if exists $k\leq n$ real roots of $\Delta$, $\la_{i}\in[-1,1]$, $i=1,\ldots,\,k$, the nonlinear sliding 
  region $\Sigma^{sl}_n\neq\emptyset$, in fact, $\Sigma^{sl}_n=\Sigma$. 
  Otherwise, $\Sigma^{sw}_{n}=\Sigma$ and the statement (a) is proved.

By the Theorem \ref{mainteo}, there exists a singular perturbation problem, 
$r\geq0$, $x\in\Sigma^{sl}_n$, given by
\begin{eqnarray}
r\dot{\theta}&=&-\sin\theta\left((b/2+d/2)+(b/2-d/2)\psi(\theta)+Q(\psi(\theta))\right),\nonumber\\
\dot{x}&=&(a+c)/2+\psi(\theta)(a-c)/2+P(\psi(\theta))\nonumber.
\end{eqnarray}

The slow manifold for $\theta\in(0,\pi)$ is determined by the equation
 $(b/2+d/2)+(b/2-d/2)\psi(\theta)+Q(\psi(\theta))=0$,
 then it is the union of $k\leq n$
  lines $\theta=\theta_i,\,i=1,\ldots,k$.  The slow flow is given by the solutions of the reduced problem represented by
\begin{eqnarray}0&=&-\sin\theta\left((b+d)+(b-d)\psi(\theta)+Q(\psi(\theta))\right)\\ \dot{x}&=&(a+c)/2+\psi(\theta)(a-c)/2+P(\psi(\theta)).\end{eqnarray}

So, for $r=0$ the dynamics on each connected component of the slow manifold is 
given by $\sgn(\dot{x})$ and (b) is proved.
\end{proof}

\noindent \textbf{Example 3.} Consider the PSVF $X_N=(X_+,X_-)$ with $X_+=(1,-1)$, $X_-=(2,-1)$ and the continuous combination 
\[\widetilde{X}=(1/2-\lambda/2 +{\lambda}^2,  1-\lambda -2 {\lambda}^2 +{\lambda}^3 ).\]
The corresponding  slow--fast system is  
\[r\dot{\theta}=-\sin\theta( 1-\lambda - 2{\lambda}^2 +{\lambda}^3),\quad \dot{x}=1/2-\lambda/2 +{\lambda}^2,\]
with $\lambda=\varphi(\cot\theta).$  Since $1-\lambda - 2{\lambda}^2 +{\lambda}^3=0$ has two solutions 
$\lambda_1\in (-1,0)$ and  $\lambda_2\in (0,1)$ the slow manifold has two 
components $\theta=\theta_1\in(0,\pi/2)$ and $\theta=\theta_2\in(\pi/2,\pi).$ \
Finally  $1/2-\lambda/2 +{\lambda}^2>0$ for $\lambda\in[-1,1]$ implies that the slow flow is equivalent to $\dot{x}=1.$
See figure \eqref{figFNsewing}.

\subsection{Saddle}
We say that $X_N=(X_+,X_-)$  is the normal form of the saddle PSVF defined on $\R^2$ if 
$X_+(x,y)=(x,-1-y)$, $X_-(x,y)=(x,1-y)$,  $h(x,y)=y$ and
 $\Sigma=h^{-1}(0)$. In this case we have $\Sigma=\Sigma^{s}$. 

\begin{proposition}\label{teo-saddlens} Let $X_N=(X_+,X_-)$ be the normal form of the saddle PSVF  and the continuous combination
\[\widetilde{X}=(P(\la)+x,-\la+Q(\la)-y).\] 

\begin{itemize}
\item [(a)] $\Sigma=\Sigma^{sl}_n$ and there exists a real continuous function $\Delta=\Delta(\la)$ such that if $\Delta$ has $k\leq n$ 
real roots $\la_{i}\in[-1,1]$, $i=1,\ldots,\,k$ then the singular perturbation problem \eqref{eqteo} in the Theorem \ref{mainteo} 
is such that the slow manifold is the union of $k$ lines $\theta=\theta_i,\,i=1,\ldots,k$.
\item [(b)] The slow flow on $\theta=\theta_i$ follows the positive direction of $x$-axis if $x>-P(\psi(\theta_i))$ and
follows the negative direction of $x$-axis if $x<-P(\psi(\theta_i))$.
\end{itemize}
\end{proposition}

\begin{proof} Let $\widetilde{X}$ be the continuous combination of $X_N$ given by the statement of the proposition. Then
$\widetilde{X}=\left(P(\la)+x,-\la+Q(\la)-y\right).$ Using our definition, the nonlinear sliding region $\Sigma_n^{sl}$ in $\Sigma$
  is given by the solutions of $ (\widetilde{X}.h)(p)=0,\, \lambda\in [-1,1]$. So, let us consider the continuous function 
  $\Delta=\Delta(\la)=-\la+Q(\la)$. Since $\Delta(1)=-1$ and $\Delta(-1)=1$, the intermediate value theorem guarantees 
  that $\Sigma=\Sigma^{sl}_n$. Now, we consider the polar blow up coordinates given by 
   $y=r\cos\theta$ and $\e=r\sin\theta$, with $r\geq0$ and $\theta\in[0,\pi]$. Using these coordinates the parameter value $\e=0$ is 
   represented by $r=0$ and the blow up induces the vector field given by
\[r\dot{\theta}=-\sin\theta\left(-\psi(\theta)+Q(\psi(\theta))-r\cos\theta\right),\quad 
\dot{x}=P(\psi(\theta))+x.\] 
The slow manifold is $\mathcal{S}=\{(\theta_i,x)\in(0,\pi)\times\mathbb{R}:-\psi(\theta_i)+Q(\psi(\theta_i))=0\}$. 
The slow flow is given by the solutions of the reduced problem represented by
\begin{equation}\label{reduced}
0=-\psi(\theta)+Q(\psi(\theta)),\quad \dot{x}=P(\psi(\theta))+x.\end{equation}
Then the reduced flow on $\theta=\theta_i$ follows the positive direction of $x$-axis if $x>-P(\psi(\theta_i))$ and 
follows the negative direction of $x$-axis if $x<-P(\psi(\theta_i))$. \end{proof}

\begin{corollary} In the conditions of the Proposition \ref{teo-saddlens}, 
let $q$ be a  normally hyperbolic singular point of the reduced problem \eqref{reduced} with $r=0$.
Then there exists $\e_0>0$ such that for $0<\e<\e_0$, $X_\e$ has a saddle point or a repelling node point near $q$. 
\end{corollary}

\begin{proof} We observe that any point $q$ on the 
slow manifold $\mathcal{S}$ is normally hyperbolic if
\[(\partial\alpha/\partial \theta)(q,0)=\dfrac{\partial}{\partial \theta}\left(-\sin\theta(-\psi(\theta)+Q(\psi(\theta)))\right)\neq0.\]
 Let us assume that, for every normally hyperbolic $q\in\mathcal{S}$, $(\partial\alpha/\partial \theta)(q,0)$ 
has $k^s$ eigenvalues with negative real part and $k^u$ eigenvalues with positive real part for the fast system. 
Lemma 14 in \cite{BPT} implies that $X_\e$ has a singular point $q_\e$ with approaches $q$ when
 $\e$ is near to zero and $q_\e$ has a $(j^s+k^s)$-dimensional local stable manifold $W_\e^s$ and $(j^u+k^u)$-dimensional 
 local unstable manifold $W_\e^u$, where $j^s$ and $j^u$ are the dimensions of the local stable and unstable 
 manifold, respectively. In this case, $j^u=1$ and $j^s=0$. 
 If $k^u=0$ and $k^s=1$, then $q_\e$ is a saddle or if $k^u=1$ and $k^s=0$, then $q_\e$ is a repelling node.
\end{proof}

\noindent \textbf{Example 4.} Consider the PSVF $X_N=(X_+,X_-)$ with $X_+=(x,-1-y)$, $X_-=(x,1-y)$ and the continuous combination 
\[\widetilde{X}=( -1+ {\lambda}^2 +x, -1-\lambda+{\lambda}^2-y).\]
The corresponding  slow--fast system is  
\[r\dot{\theta}=-\sin\theta(-1-\lambda+{\lambda}^2-r\cos\theta),\quad \dot{x}=-1+ {\lambda}^2 +x,\]
with $\lambda=\varphi(\cot\theta).$  Since ${\lambda}^2-\lambda-1=0$ has two solutions 
$\lambda_0\in (-1,0)$ and  $\lambda_1>1$ the slow manifold has only one
component $\theta=\theta_0\in(\pi/2,\pi).$ \
The slow flow $\dot{x}=x+{\lambda_0}^2-1$ has one repelling equilibrium point $x_0\approx0.6...$
See figure \eqref{figFNsewing}.

\begin{figure}
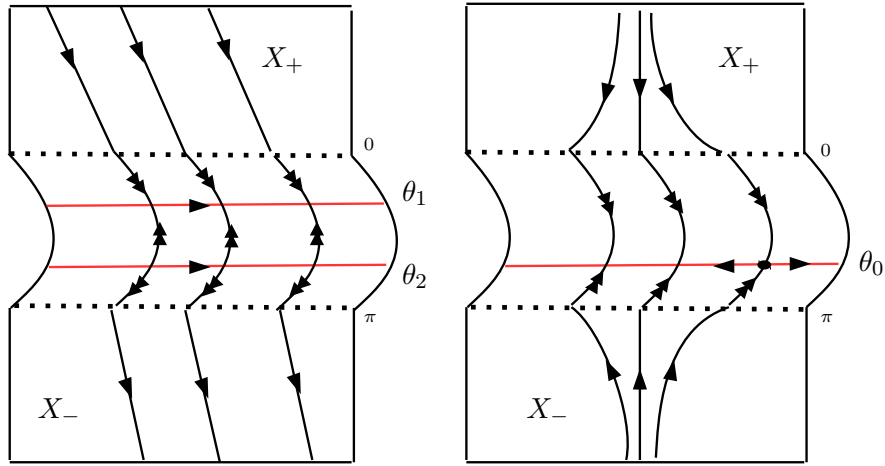

\begin{center}
\begin{overpic}[width=6cm]{DIP-fn1.pdf}
\put(61,35){\tiny$\pi$}\put(61,62){\tiny$0$}\put(10,20){$X_-$}\put(45,75){$X_+$}\put(67,41){$\theta_2$}\put(67,54){$\theta_1$}
\end{overpic}
\begin{overpic}[width=6cm]{DIP-fn2.pdf}
\put(61,35){\tiny$\pi$}\put(61,61){\tiny$0$}\put(15,20){$X_-$}\put(45,75){$X_+$}\put(67,43){$\theta_0$}
\end{overpic}
\end{center}
\caption{Slow-fast system obtained blowing up the nonlinear regularization of 
the normal forms of sewing and saddle PSVF}\label{figFNsewing}
\end{figure}

\subsection{Fold}

We say that $X_N=(X_+,X_-)$  is the normal form of the fold PSVF defined on $\R^2$ if  $X_+(x,y)=(1,x)$, $X_-(x,y)=(1,1),$
$h(x,y)=y$ and  $\Sigma=h^{-1}(0)$. In this case $\Sigma^{s}=\{(x,0)\in\Sigma:x< 0\}$.

\begin{proposition}\label{teo-foldns} Let $X_N=(X_1,\,X_2)$ be the normal form of the fold PSVF defined on $\R^2$ 
and the continuous combination
\[\widetilde{X}=(1+P(\la),1/2-\lambda/2+Q(\la)+(1/2+\la/2)x ).\] 
\begin{itemize}
\item [(a)] The singular perturbation problem \eqref{eqteo} in the Theorem \ref{mainteo} 
for this case is such that the slow manifold is a graphic $x=x(\theta)$ which is zero for $\theta=0$ and goes to $-\infty$ when $\theta$ goes to $\pi$.
\item [(b)] The slow flow is determined by $\dot{x}=1+P(\psi(\theta)).$
\end{itemize}
\end{proposition}

\begin{proof} Let $\widetilde{X}$ be the continuous combination of $X$ given by the statement of the proposition.
Then $p=(x,0)\in\Sigma^{sl}_n$ if there exists $\la\in[-1,1]$ such that $(\widetilde{X}.h)(p)=0$, so solving this we obtain 
\begin{equation}\label{eq-nsfold}
1/2-\lambda/2+Q(\la)+(1/2+\la/2)x=0.
\end{equation}  

The singular perturbation problem \eqref{eqteo} in the Theorem \ref{mainteo} for this case is 
\begin{eqnarray}\label{reduced-fold}
r\dot{\theta}&=&-\sin\theta\left(1/2-\psi(\theta)/2+Q(\psi(\theta))+(1/2+\psi(\theta)/2)x\right),\nonumber\\
\dot{x}&=&1+P(\psi(\theta))\nonumber.
\end{eqnarray}
The slow manifold is the graphic of a  function which is 0 for $\theta=0$ and goes to $-\infty$ when $\theta$ 
 goes to $\pi$. The reduced flow is determined by  $\dot{x}>0$. The fast flow satisfies  that  that $\dot{\theta}>0$, for 
 $0<\theta<M(x)$ and $\dot{\theta}<0$, for $M(x)<\theta<\pi$, with $M(x)$ given 
 implicitly by $1/2-\psi(\theta)/2+Q(\psi(\theta))+(1/2+\psi(\theta)/2)x=0$.  
\end{proof}

\noindent \textbf{Example 5.} Consider the PSVF $X_N=(X_+,X_-)$ with $X_+=(1,x)$, $X_-=(1,1)$ and the continuous combination 
\[\widetilde{X}=({\lambda}^2,  -1/2   -\la/2+{\lambda}^2+(\la/2+1/2)x).\]
The corresponding  slow--fast system is  
\[r\dot{\theta}=-\sin\theta( -1/2   -\la/2+{\lambda}^2+(\la/2+1/2)x),\quad \dot{x}={\lambda}^2,\]
with $\lambda=\varphi(\cot\theta).$  
The equation $-1/2   -\la/2+{\lambda}^2+(\la/2+1/2)x=0$ defines the slow manifold implicitly 
$x=x(\theta)$ as a smooth function satisfying that $x(0)=0$ and $\lim_{x\rightarrow\pi}x(\theta)=-\infty.$\
The slow flow is locally equivalent to $\dot{x}=1.$ Moreover there exists one non normally hyperbolic point on the slow manifold.
See figure \eqref{figFNfold}.

\section{Nonlinear regularization of Codimension One PSVF on $\R^2$}\label{s5}

The codimension one singularities of the PSVF are those which generically occur in one parameter families of PSVF. The classification of codimension one local bifurcations was achieved by \cite{GST, KRG, T}.  In what follows we discuss about the  normal forms $X_N$ of codimension one PSVF defined on $\R^2$. 
Saddle-node, elliptical fold, hyperbolic fold and parabolic fold are considered.  
For each one of these normal forms we consider a continuous combination 
$\widetilde{X}$  given by \eqref{contcomb} and analyze the dynamics of their 
nonlinear regularizations using singular perturbation techniques.

\begin{figure}[h!]
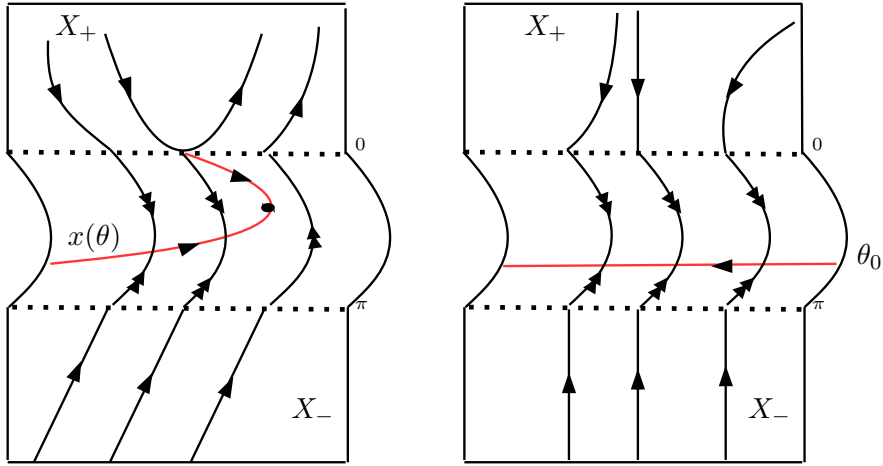

\begin{center}
\begin{overpic}[width=6cm]{DIP-fn3.pdf}
\put(60,37){\tiny$\pi$}\put(60,62){\tiny$0$}\put(50,20){$X_-$}\put(13,80){$X_+$}\put(15,47){$x(\theta)$}
\end{overpic}
\begin{overpic}[width=6cm]{DIP-fn4.pdf}
\put(60,37){\tiny$\pi$}\put(60,62){\tiny$0$}\put(50,20){$X_-$}\put(15,80){$X_+$}\put(67,44){$\theta_0$}
\end{overpic}
\end{center}
\caption{Slow-fast system obtained blowing up the nonlinear regularization of 
the normal forms of fold and saddle-node PSVF.}\label{figFNfold}
\end{figure}

\subsection{Saddle-node}

We say that $X_N=(X_1,X_2)$  is the normal form of the saddle-node PSVF defined on $\R^2$ if  $X_+(x,y)=(-x^2,-1)$, $X_-(x,y)=(0,1)$,
$h(x,y)=y$ and $\Sigma=h^{-1}(0)$. In this case $\Sigma^s=\Sigma$.

\begin{proposition}\label{nssaddle-node} Let $X_N=(X_+,\,X_-)$ be the normal form of
the saddle-node PSVF and the continuous combination 
  \[\widetilde{X}=(  P(\la)+(-1/2-\la/2)x^2,-\la+Q(\la)).\]
\begin{itemize}
\item [(a)] There exists a real continuous function $\Delta=\Delta(\la)$ such
 that if $\Delta$ has $k\leq m$ real roots $\la_{i}\in[-1,1]$, $i=1,\ldots,\,k$ then $\Sigma=\Sigma^{sl}_n$. 
\item [(b)] The singular perturbation problem \eqref{eqteo} in the Theorem \ref{mainteo}
 is such that the slow manifold is the union of $k$ lines $\theta=\theta_i,\,i=1,\ldots,k\leq m$. 
 Moreover the slow flow on each $\theta=\theta_i$ is determined by \linebreak $\dot{x}=\sgn\left(P(\psi(\theta))-(x^2/2)(1+\psi(\theta))\right)$.
\item[(c)] For each $i=1,\ldots,k\leq m$, the number of singular points of the slow flow on $\theta=\theta_i$ is at most two whenever $P(\psi(\theta_i))\geq0$.
\end{itemize}
\end{proposition}

\begin{proof} The items (a) and (b) follow using the same ideas as in the proofs of the previous theorems. 
For the item (c), we consider the singular perturbation problem \eqref{eqteo}
\begin{equation}\label{reduced-saddle-no}
r\dot{\theta}=-\sin\theta\left(-\psi(\theta)+Q(\psi(\theta))\right),\quad \dot{x}=P(\psi(\theta))+(-1/2-\psi(\theta)/2)x^2.
\end{equation}
The there are two singular points at $(\theta,x)=\left(\theta_i,\pm\sqrt{2P(\psi(\theta_i))/(1+\psi(\theta_i))}\right)$, if  $P(\psi(\theta_i))\geq 0$, for each $i=1,\ldots,k$.
\end{proof}

\noindent \textbf{Example 6.} Consider the PSVF $X_N=(X_+,X_-)$ with $X_+=(-x^2,-1)$, $X_-=(0,1)$ and the continuous combination 
\[\widetilde{X}=(-1 +{\lambda}^2+(-1/2-\la/2)x^2, -1-\lambda+ {\lambda}^2).\]
The corresponding  slow--fast system is  
\[r\dot{\theta}=-\sin\theta( -1-\lambda+ {\lambda}^2),\quad \dot{x}=-1 +{\lambda}^2+(-1/2-\la/2)x^2,\]
with $\lambda=\varphi(\cot\theta).$  The equation $-1-\lambda+ {\lambda}^2=0$ has only one solution $\lambda_0\in(-1,0)$ and consequently the
 slow manifold is $\theta=\theta_0$ with $\theta_0\in(\pi/2,\pi).$
Since $-1 +{\lambda_0}^2+(-1/2-\la_0/2)x^2<0$ the slow flow is locally equivalent to $\dot{x}=-1.$
See figure \eqref{figFNfold}.

\subsection{Elliptical fold}
We say that $X_N=(X_+,X_-)$  is the normal form of the elliptical fold PSVF defined on $\R^2$ if  $X_+(x,y)=(1,-x)$, $X_-(x,y)=( -1,-x)$,
$h(x,y)=y$ and $\Sigma=h^{-1}(0)$.  

\begin{proposition}Let $X_N=(X_+\,X_-)$ be the normal form of the elliptical fold PSVF and  the continuous combination
 \[\widetilde{X}=(\la+ P(\la),Q(\la)-x) .\] 
 
\begin{itemize}
\item[(a)] The singular perturbation problem \eqref{eqteo} in the Theorem \ref{mainteo}, for this case, is such that the slow manifold is a graphic $x=Q(\psi(\theta)$ which joins $(0,0)$ with $(\pi,0)$. Moreover, it is possible that 
the curve $x=Q(\psi(\theta)$ has points $(\theta_0,0)$ such that $Q(\psi(\theta_0)=0$, $\theta_0\in(0,\pi)$.

\item[(b)] The reduced flow is determined by $\dot{x}=\sgn(\psi(\theta_0)+P(\psi(\theta_0))$, $\theta_0\in(0,\pi)$.
\end{itemize}
\end{proposition}

\noindent \textbf{Example 7.} Consider the PSVF $X_N=(X_+,X_-)$ with $X_+=(1,-x)$, $X_-=(-1,-x)$ and the continuous combination 
\[\widetilde{X}=(-1+\lambda +{\lambda}^2, -1-x+{\lambda}^2).\]
The corresponding  slow--fast system is  
\[r\dot{\theta}=-\sin\theta( -1-x+{\lambda}^2),\quad \dot{x}=-1+\lambda +{\lambda}^2,\]
with $\lambda=\varphi(\cot\theta).$  The equation $-1-x+{\lambda}^2=0$ defines implicitly the
slow manifold as a smooth curve connecting $\theta=0$ and $\theta=\pi$.
The slow flow is determined by $ \dot{x}=-1+\lambda +{\lambda}^2$. Thus $\dot{x}>0$ if $\lambda\in(\lambda_0,1)$ and 
$\dot{x}<0$ if $\lambda\in(-1,\lambda_0)$, where $\lambda_0=\varphi(\cot\theta_0), \theta_0\in(0,\pi/2).$ 
See figure \eqref{figFNel}.

\begin{figure}[h!]
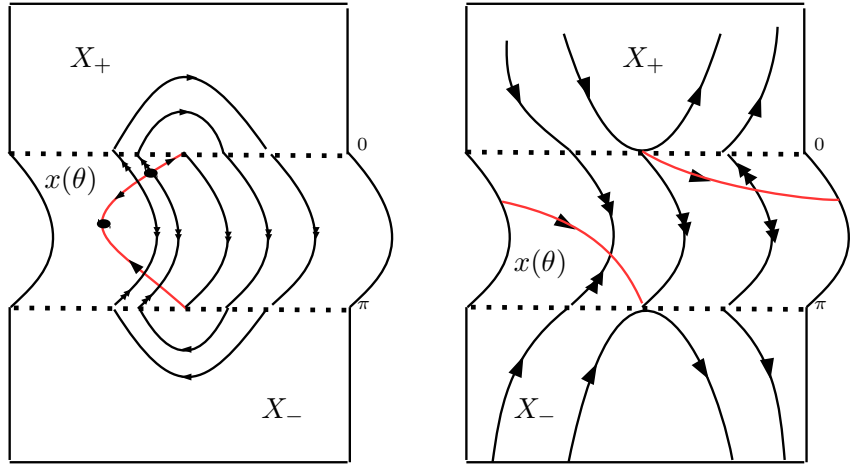

\begin{center}
\begin{overpic}[width=6cm]{DIP-fn5.pdf}
\put(60,37){\tiny$\pi$}\put(60,62){\tiny$0$}\put(45,20){$X_-$}\put(15,75){$X_+$}\put(11,56){$x(\theta)$}
\end{overpic}
\begin{overpic}[width=6cm]{DIP-fn6.pdf}
\put(60,37){\tiny$\pi$}\put(60,62){\tiny$0$}\put(13,20){$X_-$}\put(30,75){$X_+$}\put(13,43){$x(\theta)$}
\end{overpic}
\caption{Slow-fast system obtained blowing up the nonlinear regularization of 
the normal forms of elliptical fold and hyperbolic fold PSVF.}\label{figFNel}
\end{center}
\end{figure}

\subsection{Hyperbolic fold}

We say that $X_N=(X_+,X_-)$  is the normal form of the hyperbolic fold PSVF defined on $\R^2$ if 
$X_+(x,y)=(1,x)$, $X_-(x,y)=(1,-2x)$,  $h(x,y)=y$ 
and $\Sigma=h^{-1}(0)$.

\begin{proposition} \label{teo-hyperbolic-2} Let $X_N^{-}$ be the normal form of the hyperbolic fold and the continuous combination
 \[\widetilde{X}=(1+P(\la),   Q(\la)+(-1/2+3\la/2)x).\] 
 The singular perturbation 
problem \eqref{eqteo} in the Theorem \ref{mainteo} is such that the reduced flow has only two 
singular points, $(0,0)$ and $(\pi,0)$, and the slow manifold is the curve $\frac{x}{2}(1-3\psi(\theta))=Q(\psi(\theta))$
that joins the two folds and tends to $\pm\infty$ if $\theta$ tens to $\theta_0$, where $\psi(\theta_0)=1/3$. 
Moreover, the reduced flow is singular.
\end{proposition}

\noindent \textbf{Example 8.} Consider the PSVF $X_N=(X_+,X_-)$ with $X_+=(1,x)$, $X_-=(1,-2x)$ and the continuous combination 
\[\widetilde{X}=({\lambda}^2, -1+{\lambda}^2+ (-1/2+3\la/2)x).\]
The corresponding  slow--fast system is  
\[r\dot{\theta}=-\sin\theta( -1+{\lambda}^2+ (-1/2+3\la/2)x),\quad \dot{x}={\lambda}^2,\]
with $\lambda=\varphi(\cot\theta).$  The equation $-1+{\lambda}^2+ (-1/2+3\la/2)x=0$ defines implicitly the
slow manifold as a pair of smooth curves with an asymptote on $\theta_0$ with 
$\varphi(\cot\theta_0)=1/3.$
The slow flow is determined by $ \dot{x}={\lambda}^2$ , thus it is locally equivalent to $\dot{x}=1.$
See figure \eqref{figFNel}.

\subsection{Parabolic fold}

We say that $X_N=(X_+,X_-)$  is the normal form of the parabolic fold PSVF defined on $\R^2$ if 
$X_+(x,y)=(-1,-x)$, $X_-(x,y)=(1,2x)$,   $h(x,y)=y$ 
and $\Sigma=h^{-1}(0)$.

\begin{figure}[h!]
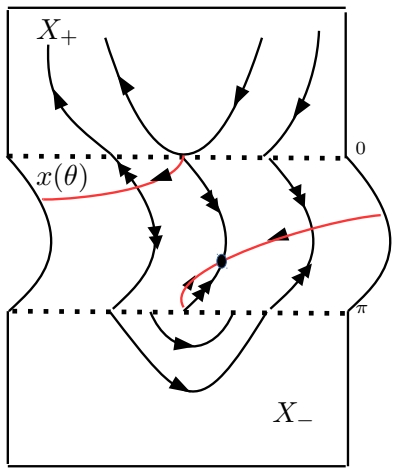

\begin{center}
\begin{overpic}[width=6cm]{DIP-fn7.pdf}
\put(60,37){\tiny$\pi$}\put(60,62){\tiny$0$}\put(47,20){$X_-$}\put(10,80){$X_+$}\put(10,57){$x(\theta)$}
\end{overpic}
\caption{Slow-fast system obtained blowing up the nonlinear regularization of 
the normal form of parabolic fold PSVF.}\label{figPar}
\end{center}
\end{figure}

\begin{proposition}\label{teo-parabolic} Let  $X_N=(X_+,X_-)$  be the normal form of the parabolic fold PSVF and the continuous combination
\[\widetilde{X}=(-\la+P(\la),     Q(\la) +(1/2-3\la/2)x )\] 
 The singular perturbation problem \eqref{eqteo} is \[r\dot{\theta}=-\sin\theta\left(Q(\psi(\theta))+(1/2-3\psi(\theta)/2)x\right),\quad \dot{x}=-\psi(\theta)+P(\psi(\theta)).\] 
 The slow manifold  a graphic $x=x(\theta)$ given implicitly by  $x(3\psi(\theta)-1)=2Q(\psi(\theta))$, satisfying that  $x(0)=x(\pi)=0$ 
and tends to $\pm\infty$ if $\theta$ tends to $\theta_0$ with $\psi(\theta_0)=1/3$. 
 The reduced flow is given by $\dot{x}=\sgn(-\psi(\theta_0)+P(\psi(\theta_0)))$. 
\end{proposition}

\noindent \textbf{Example 9.} Consider the PSVF $X_N=(X_+,X_-)$ with $X_+=(-1,-x)$, $X_-=(1,2x)$ and the continuous combination 
\[\widetilde{X}=(-1-\lambda+{\lambda}^2 ,-1+{\lambda}^2+(1/2-3\la/2)x).\]
The corresponding  slow--fast system is  
\[r\dot{\theta}=-\sin\theta( -1+{\lambda}^2+(1/2-3\la/2)x),\quad \dot{x}=-1-\lambda+{\lambda}^2,\]
with $\lambda=\varphi(\cot\theta).$  The equation $ -1+{\lambda}^2+(1/2-3\la/2)x=0$ defines implicitly the
slow manifold as a smooth curve  with an asymptote on $\theta_0$ with 
$\varphi(\cot\theta_0)=1/3.$
The slow flow is determined by $\dot{x}=-1-\lambda+{\lambda}^2$. Thus $\dot{x}<0$ if $\lambda\in(\lambda_0,1)$ and 
$\dot{x}>0$ if $\lambda\in(-1,\lambda_0)$, where $\lambda_0=\varphi(\cot\theta_0), \theta_0\in(\pi/2,\pi).$
See figure \eqref{figPar}. 

\section{Some Examples of PSVF's on $\R^3$}\label{s5}

In this section we present some examples of nonlinear regularization of PSVF's on $\R^3$. 
In the previous section, the continuous combinations of $X_+$ and $X_-$ had nonlinear terms depending only on $\lambda$. 
Indeed, in all the examples we considered $ \widetilde {X}$ of the kind \eqref{contcomb}. \\

Now, in order to obtain some generic situations, the examples of continuous combinations of discontinuous systems in $\R^3$ 
will have the nonlinear terms  also  depending on $p\in\Sigma$, with $\Sigma=\{(x,y,0)\in\mathbb{R}^3\}$. \\

\noindent \textbf{Example 10.} Consider the PSVF $X=(X_+,X_-)$ with $X_+=(0,0,-1)$, $X_-=(0,1,1)$ and the continuous combination 
\[\widetilde{X}=(P(\la,x,y,z),1/2-\la/2+Q(\la,x,y,z),-\la+R(\la,x,y,z)),\] where $P(\la,x,y,z)=x(\la^2-1)$, $Q(\la,x,y,z)=-y/2(\la^2-1)$ and $R(\la,x,y,z)=-\la(\la^2-1)$ 
are continuous polynomials satisfying the condition $P(\pm 1,x,y,z)=Q(\pm 1,x,y,z)=R(\pm 1,x,y,z)=0$.  Thus we have \[\widetilde{X}=(x(\la^2-1),1/2-\la/2-y/2(\la^2-1), -\la-\la(\la^2-1)).\] 

The nonlinear sliding region, $\Sigma_n^{sl}$, is $\Sigma$ and for any $p\in\Sigma_n^{sl}$ the nonlinear sliding vector field is given by \[\dot{x}=-x,\,\dot{y}=1/2+y/2.\]

Now, let us consider the nonlinear regularization given by \linebreak $X^{\e}=\widetilde{X}(\varphi (z/\e),x,y,z)$ and apply Theorem (2.1).
We get the singular perturbation problem

\[\begin{array}{lll}
r\dot{\theta}&=&-\sin\theta(-\psi(\theta)^3)\\
 \dot{x}&=&x(\psi(\theta)^2-1)\\
  \dot{y}&=&(1+y)/2-\psi(\theta)/2-y\psi(\theta)^2/2.
\end{array}\]

The slow manifold is $\theta=\pi/2$ and the corresponding  slow system is determined by $\dot{x}=-x,\quad \dot{y}=1/2+y/2$, and it has a saddle point in $(0,-1)$. See figure \eqref{fig:sewing-saddle}.\\

\begin{figure}[h!]
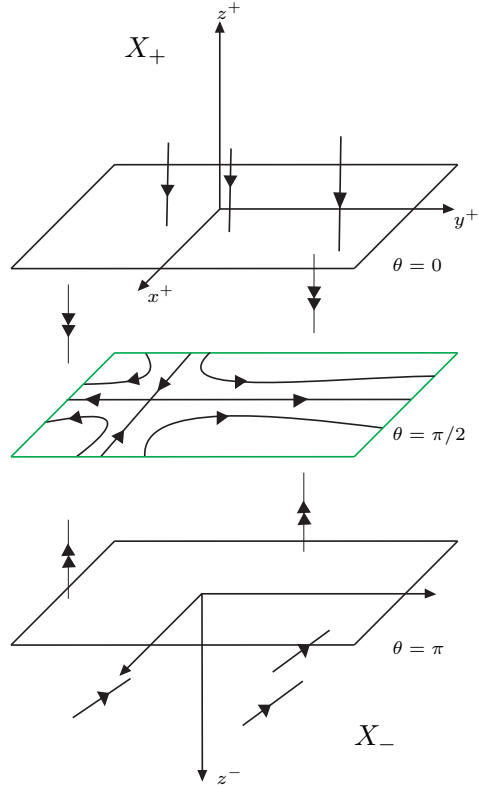

\begin{center}
\begin{overpic}[width=6cm]{sewingR3.pdf}
\put(50,17){\tiny$\theta=\pi$}\put(50,67){\tiny$\theta=0$}\put(45,5){$X_-$}\put(15,95){$X_+$}\put(50,45){\tiny $\theta=\pi/2$}\put(27,100){\tiny$z^+$}\put(58,73){\tiny$y^+$}\put(18,63){\tiny$x^+$}\put(27,0){\tiny$z^-$}
\end{overpic}
\caption{Slow-fast system obtained blowing up the nonlinear regularization of 
a sewing PSVF on $\R^3$.}\label{fig:sewing-saddle}
\end{center}
\end{figure}

\noindent \textbf{Example 11.} Consider the PSVF $X=(X_+,X_-)$ with $X_+=(1,0,x)$, $X_-=(0,0,1)$. The regular points in $\Sigma$ are $\{(x,y,0)\in\mathbb{R}^3:\,x\neq 0\}$
(sewing if $x>0$ and sliding if $x<0$). The singular points in $\Sigma$ are $\{(0,y,0)\in\mathbb{R}^3\}$.

Let us consider a continuous combination of $X_+$ and $X_-$ given by
\[\widetilde{X}=\left(1/2+\la/2+P,Q,\dfrac{x+1}{2}+\dfrac{x-1}{2}\la+R\right),x\neq 0, \]
where  
\[\begin{array}{ll}
P=P(\la,x,y,z)=(\la^2-1)(-2+2x+x^2-x^3)/4x,\\
Q=Q(\la,x,y,z)=(\la^2-1)(y-2xy+x^2y)/4x,\\
R=R(\la,x,y,z)=0.
\end{array}\]
Solving $$\dfrac{x+1}{2}+\dfrac{x-1}{2}\la=0$$ in the variable $\la$ we have that the nonlinear sliding region is $$\Sigma^{sl}_n=\{(x,y,0)\in\mathbb{R}^3:\,x<0\}.$$ The sliding vector field is given by $X^\la=(-x-1,y,0).$\\

Let $X^{\e}=\widetilde{X}(\varphi (z/\e),x,y,z)$ be the $\varphi$-nonlinear regularization. For any $p\in\Sigma^{sl}_n$, the associated singular perturbation problem  is

\begin{displaymath}
\begin{array}{lcl}
r\dot{\theta}&=&-\sin\theta\left( \dfrac{x+1}{2}+\dfrac{x-1}{2}\psi(\theta)\right),\\
\dot{x}&=&1/2+\psi(\theta)/2+\dfrac{-2+2x+x^2-x^3}{4x}(\psi(\theta)^2-1),\\
 \dot{y}&=&\dfrac{y-2xy+x^2y}{4x}(\psi(\theta)^2-1),
\end{array}
\end{displaymath}
with $\psi(\theta)=\varphi(\cot\theta)$.  The slow manifold is the curve implicitly defined  by \[\psi(\theta)=-\dfrac{x+1}{x-1},\] and 
the slow flow is determined by $\dot{x}=-x-1,\quad \dot{y}=y$. In this case, the slow flow has a saddle in $(-1,0)$. See figure \eqref{fig:fold-reg-space}.\\

\begin{figure}[h!]
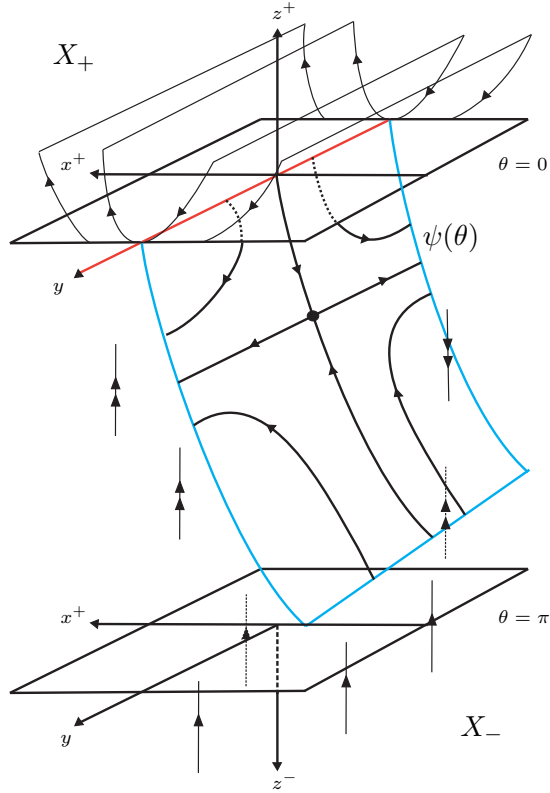

\begin{center}
\begin{overpic}[width=7cm]{fold-reg-space.pdf}
\put(65,20){\tiny$\theta=\pi$}\put(65,80){\tiny$\theta=0$}\put(60,5){$X_-$}\put(6,94){$X_+$}\put(55,70){$\psi(\theta)$}\put(35,100){\tiny$z^+$}\put(6,64){\tiny$y$}\put(7,4){\tiny$y$}\put(7,20){\tiny$x^+$}\put(7,80){\tiny$x^+$}\put(35,-2){\tiny$z^-$}
\end{overpic}
\caption{Slow-fast system obtained blowing up the nonlinear regularization of 
the normal form fold regular PSVF on $\R^3$.}\label{fig:fold-reg-space}
\end{center}
\end{figure}

Now, take the continuous combination $\widetilde{X}$ of $X_+$ and $X_-$ with 
\[\begin{array}{ll}
P=P(\la,x,y,z)=-(\la^2-1)(1+x)(-1+y+xy)/4x,\\

Q=Q(\la,x,y,z)=(\la^2-1)(x+1)^2,\\
R=R(\la,x,y,z)=\dfrac{x+1}{2}(\la^2-1).
\end{array}\]
By solving $$\dfrac{x+1}{2}+\dfrac{x-1}{2}\la+\frac{x+1}{2}(\la^2-1)=0$$ in the variable $\la$, we find  that 
$\Sigma^{sl}_n=\Sigma\setminus\{x=0\}$; besides on $\{(x,y,0)\in\Sigma:\,x>0\}$ we have two nonlinear sliding vector field and on $\{(x,y,0)\in\Sigma:\,x<0\}$ only one.

The sliding vector field $X^{\lambda_1}$ is 
\[\dot{x}=1/2+\dfrac{(1+x)(-1+y+xy)}{4x},\quad  \dot{y}=-(1+x)^2,\quad \dot{z}=0;\]
and the sliding vector field $X^{\lambda_2}$ is 
\[\dot{x}=y, \quad
 \dot{y}=-4x, \quad
\dot{z}=0.\]
%
%
%
The $\varphi$-nonlinear regularization $X^{\e}=\widetilde{X}(\varphi (z/\e),x,y,z)$. For any $p\in\Sigma^{sl}_n$, the singular perturbation problem associated is 

\begin{displaymath}
\begin{array}{lcl}
r\dot{\theta}&=&-\sin\theta\left( \psi(\theta)\left(\dfrac{x-1}{2}+\dfrac{x+1}{2}\psi(\theta)\right)\right)\\
\vspace{0.2cm}
\dot{x}&=&1/2+\psi(\theta)/2+\dfrac{-(1+x)(-1+y+xy)}{4x}(\psi(\theta)^2-1)\\
\vspace{0.2cm}
 \dot{y}&=&(x+1)^2(\psi(\theta)^2-1)
\end{array}
\end{displaymath}
where $\psi(\theta)=\varphi(\cot\theta)$. Then the slow manifold is the curve given by the plane $\theta=\pi/2$ and $\psi(\theta)=\frac{1-x}{1+x}$, $x>0$. The slow flow on $\theta=\pi/2$ is given by 

\[\dot{x}=1/2+\dfrac{(1-x)(-1+y+xy)}{4x}, \quad
 \dot{y}=-(x+1)^2,\]
and for $\psi(\theta)=\dfrac{1-x}{1+x}$, the slow flow is 
\[\dot{x}=y, \quad 
 \dot{y}=-4x.\]
The slow-fast system obtained blowing up the nonlinear regularization is shown in figure \eqref{fig:fold-regular-cruz}.

\begin{figure}[h!]
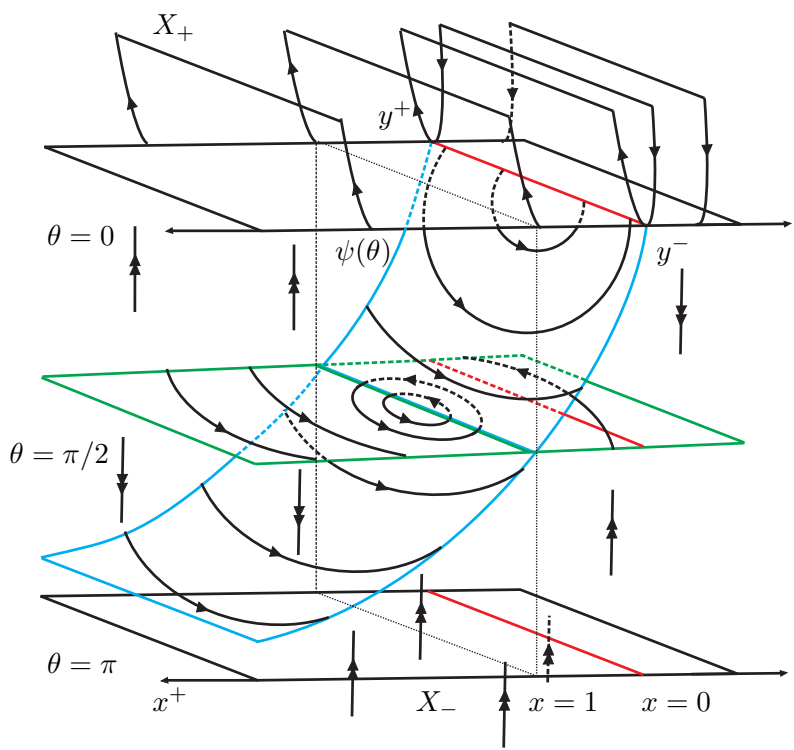

\begin{center}
\begin{overpic}[width=10cm]{fold-cruz.pdf}
\put(1,10){$\theta=\pi$}\put(1,67){$\theta=0$}\put(50,5){$X_-$}\put(15,95){$X_+$}\put(38,65){ $\psi(\theta)$}\put(45,83){$y^+$}\put(82,65){$y^-$}\put(15,5){$x^+$}\put(-4,38){$\theta=\pi/2$}
\put(65,5){$x=1$}\put(80,5){$x=0$}\end{overpic}
\end{center}
\caption{Slow-fast system obtained blowing up the nonlinear regularization of 
the regular fold PSVF on $\R^3$.}\label{fig:fold-regular-cruz}
\end{figure}

\section{Acknowledgments} The first  author is partially supported by  CAPES, CNPq, 
FAPESP, FP7-PEOPLE-2012-IRSES 318999, and PHB 2009-0025-PC.  The second author is partially supported by PNPD-CAPES scholarship.  The third author is partially supported by FAPESP 2016/11471-2. All authors thank the hospitality of CRM  (Centre de Recerca Matem\`{a}tica, Barcelona) during our visit in April 2016.

\end{document}